\documentclass[11pt,a4paper]{amsart}

\usepackage{amsmath}
\usepackage{amssymb}
\usepackage{enumerate}

\usepackage[T1]{fontenc}
\usepackage{lmodern}
\usepackage{textcomp}
\usepackage{tikz}               

\usetikzlibrary{calc,through,backgrounds,arrows}
\usepackage{color}

\definecolor{webblue}{rgb}{0, 0, 1.0}  
\definecolor{webred}{rgb}{1.0, 0, 0}   

\newcommand{\gauss}[3]{\genfrac{[}{]}{0pt}{}{#1}{#2}_{#3}}

\newcommand{\eye}{\ensuremath{\!\!\begin{array}{c}\frown\\[-10pt]\smile\end{array}\!\!}}
\newcommand{\nei}{\ensuremath{\not\!\!\!\!\!\begin{array}{c}\frown\\[-11pt]\smile\end{array}\!\!}}

\newcommand{\vek}[1]{\boldsymbol{#1}}

\DeclareMathOperator{\PG}{PG} \DeclareMathOperator{\GR}{GR}
\DeclareMathOperator{\PHG}{PHG} 
  
\DeclareMathOperator{\Aut}{Aut} \DeclareMathOperator{\rad}{Rad}

\usepackage[papersize={165mm,235mm},top=26mm,left=16mm,text={128mm,192mm},%
footskip=4pt,headsep=6mm,marginparsep=0pt]{geometry}
\usepackage{url}
\newtheorem{theorem}{Theorem}[section]

\newtheorem{lemma}[theorem]{Lemma}

\theoremstyle{definition}

\numberwithin{equation}{section}

\title[Analogue of the Erd\H{o}s-Ko-Rado Theorem]{A Chain Ring Analogue of the Erd\H{o}s-Ko-Rado Theorem}

\author[I. Landjev]{Ivan Landjev}
\address[I. Landjev]{Institute of Mathematics and Informatics, BAS, 
	8,~Acad. G. Bonchev. str., 1113 Sofia, Bulgaria}
\email{\tt ivan@math.bas.bg}

\author[E. Rogachev]{Emiliyan Rogachev}

\address[E. Rogachev]{Sofia University, Faculty of Mathematics and Informatics, 5,~James Bourchier Blvd, 1164 Sofia, Bulgaria}
\email{{\tt erogachev@fmi.uni-sofia.bg}}

\author[A. Rousseva]{Assia Rousseva}

\address[A. Rousseva]{Sofia University, Faculty of Mathematics and Informatics, 5,~James Bourchier Blvd, 1164 Sofia, Bulgaria}
\email{{\tt assia@fmi.uni-sofia.bg}}

\begin{document}
	
	\begin{abstract}
		In this paper, we prove an analogue of the Erd\H{o}s-Ko-Rado theorem
intersecting families of subspaces in projective Hjelmslev
geometries over finite chain rings of nilpotency index 2.
We give an example of maximal families that 
are not canonically intersectng.\\
		

	\end{abstract}

	\maketitle


\section{Some Classic Results}

In 1961 Pal Erd\H{o}s, Chao Ko and Richard Rado published a theorem which 
solved 
a problem in extremal set theory and initiated a lot of research \cite{EKR61}.

\begin{theorem}
\label{thm:EKR}(Erd\H{o}s-Ko-Rado)\ 
Let $\Omega$ be a finite set with $n$ elementsand let
$k\le n/2$ be an integer. If $\mathcal{F}$ is a family of
$k$-element subsets of 
$\Omega$ that are pairwise not disjoint then
\[|\mathcal{F}|\le{n-1\choose k-1}.\]
\end{theorem}

In the case of $n\ge 2k+1$ the upper bound is attained iff $\mathcal{S}$ is
canonically intersecting family, i.e.
it contains all $k$-element subsets of $\Omega$ containig a fixed element of
$\Omega$. In the case of $n=2k$ we have plenty of examples. All they are
constructed by splitting all $k$-element subsets in pairs
$(X,\Omega\setminus X)$ and taking one set from each such pair.

An important generalization of this problem is to find the maximal 
size of a $t$-intersecting family defined as a family of subsets
every two of which meet in at least $t$ elements.
In this case we have the following theorem.

\begin{theorem}
\label{thm:EKR1}
Let $\Omega$ be a finite set with $|\Omega|=n$ and let $k$ and $t$,
$1\le t<k$, be integers. Let $\mathcal{F}$ be a $t$-intersecting family of $k$-subsets
of $\Omega$. There exists a constant $f(k,t)$ such that if $n>f(k,t)$
\[|\mathcal{F}|\le{n-t\choose k-t},\]
and $\mathcal{F}$ meets the bound iff it is canonically intersecting.
\end{theorem}

During the years there have been various estimates for $f(k,t)$.
The best one $f(k,t)=(t+1)(k-t+1)$ is due to Frankl~\cite{F78}
and Wilson \cite{W84}. We mention here also the important result by Hilton 
and Milner \cite{HM67} which describes the largest 1-intersection systems 
that are not canonnically 1-intersecting.

\begin{theorem}
\label{thm:HM}
Let $\Omega$ be a set of cardinality $n$ and let $S$ be an 
Erd\H{o}s-Ko-Rado set of $k$-element subsets of $\Omega$,
$k\ge3$, $n\ge2k+1$. If $\mathcal{F}$ is not canonically intersecting
then
\[|\mathcal{F}|\le {n-1\choose k-1}-{n-k-1\choose k-1}+1.\]
Equality holds if and only if
\begin{enumerate}[$\bullet$]
\item $\mathcal{F}$ is is the union of some fixed $k$-subset $F$
and the set of all k-subsets $G$ of $\Omega$ containig some fixed
element $x\not\in F$ and satisfying $F\cap G\ne\varnothing$, or else
\item $k=3$ and $\mathcal{F}$ is the set of all 3-element subsets  
having an intersection of size at least 2 with a fixed 3-subset $F$.
\end{enumerate}
\end{theorem}

The Erd\H{o}s-Ko-Rado problem has been generalized
to many other structures \cite{GM16}. In 1975 Hsieh \cite{H75}
proved a generalization which is now known as a $q$-analogue of
the Erd\H{o}s-Ko-Rado theorem. There have been many improvements to Hsieh's
result. We give here a theorem which combines the results of Hsieh \cite{H75},
Frankl and Wilson \cite{FW86}, and Tanaka \cite{T06}.
In the next two theorems by dimension we mean projective dimension.

\begin{theorem}
\label{thm:hsieh}
Let $t$ and $k$ be integers with $0\le t\le k$, and
let $\mathcal{F}$ be a set of $(k-1)$-dimensional subspaces in $\PG(n-1,q)$
pairwise intersecting in at least a $(t-1)$-dimensional subspace.
If $n\ge2k$, then
$\displaystyle |\mathcal{F}|\le\gauss{n-t}{k-t}{q}$. Equality holds if and
only if  $\mathcal{F}$ is the set of all $(k-1)$-dimensional subspaces, containing a fixed 
$(t-1)$-dimensional subspace of $\PG(n-1,q)$, 
or, for $n=2k$, $\mathcal{F}$ can be also the set of all
$(k-1)$-dimensional subspaces in a fixed $(2k-t-1)$-dimensional subspace.

In case of $2k-t\le n\le 2k$, we have that
$\displaystyle |\mathcal{F}|\le\gauss{2k-t}{k}{q}$. Equality holds if and only if $\mathcal{F}$
is the set of all $(k-1)$-dimensional subspaces in a fixed
$(2k-t-1)$-dimensional  subspace.
\end{theorem}

In what follows, we shall also need a result proved recently by Tanaka~\cite{T06}.
Consider the projective geometry $\PG(n-1,q)$. Let $d, e$ be two integers
with $d+e=n-1$, $e\ge d$. Fix a subspace $W$ with $\dim W=e-1$ and
denote by $\mathcal{U}$ the set of all subspaces $\gamma$ in $\PG(n-1,q)$
with $\dim\gamma=d-1$, $\gamma\cap W=\varnothing$.

\begin{theorem}
\label{thm:tanaka}
Let $0\le t\le d$ be an integer and let $\mathcal{F}$
be a family of subspaces from $\mathcal{U}$
with $\dim(\gamma\cap\delta)\ge t-1$ for every two $\gamma, \delta\in\mathcal{U}$.
Then
\[|\mathcal{F}|\le q^{(d-t)e}\]
and equality holds if

\begin{enumerate}[(a)]
\item $\mathcal{F}$ consists of all subspaces $\gamma$
through a fixed $(t-1)$-dimensional subspace $U$ with
$U\cap W=\varnothing$;
\item in case of $e=d$, $\mathcal{F}$ is the set of all elements of 
$\mathcal{U}$ contained in a fixed $(2d-t-1)$-dimensional subspace $U$
with $\dim U\cap W=d-t-1$.
\end{enumerate}
\end{theorem}

In this note we prove analogues of the Erd\H{o}s-Ko-Rado theorem for 
intersecting families of subspaces in projective Hjelmslev geometries over finite chain rings.

\section{Preliminaries}

In this section we recall some basic facts on finite chain rings, modules over finite chain rings and finite projective Hjelmslev geometries. For a more in-depth introduction we refer to \cite{TYL99,McD74,N08}.

\subsection{Finite Chain Rings}

A ring (associative with identity $1\ne0$) is called a chain ring
if the lattice of its left (right) ideals ordered by inclusion is a chain.
Though chain rings are not necessarily commutative, it is
known that every left ideal is also a right ideal.
The largest non-trivial ideal in $R$ will be denoted hence forth by $N$.
The chain of the right ideals of $R$ can be written as
\[R>N>\ldots>N^{m-1}>(0).\]
The ideals $N^i$ are generated by the powers of a single element $\theta$,
where $\theta$ can be taken as any element in $N\setminus N^2$.
Thus $N^i=R\theta^i=\theta^iR$.
The integer $m$ is the nilpotency index or the length of the chain ring $R$.
Clearly, $R/N\cong\mathbb{F}_q$ for some prime power $q$, and
$|R|=q^m$. Throughout this paper $q, m, N$, and $\theta$ will be used with this
meaning.

Given a chain ring $R$, we fix a set with the property
that no two of its elements are congruent modulo $N$:
\[\Gamma=\{\gamma_0=0,\gamma_1=1,\gamma_2,\ldots,\gamma_{q-1}\},\]
$\gamma_i\in R$, $\gamma_i\not\equiv\gamma_j\pmod{N}$. For fixed $\Gamma$ and 
$\theta$, every element $r\in R$  is represented uniquely as
\[r=r_0+r_1\theta+\ldots+r_{m-1}\theta^{m-1},\]
where $r_0,r_1,\ldots,r_{m-1}$ are from $\Gamma$. By $\eta_i$, $0\le i\le m-1$,
we denote the natural homomorphisms
\[\eta_i:\left\{
\begin{array}{ccc}
R &\to & R/N^i \\
\sum_{j=0}^{m-1} r_j\theta^j &\to & \sum_{j=0}^{i-1} r_j\theta^j + N^i
\end{array}\right..\]

The simplest non-trivial chain rings are those withe one non-trivial ideal, 
the chain rings with $m=2$. In this case we have a complete classification of these 
rings. They are:

\begin{enumerate}[(a)]
\item for every $\sigma\in\Aut(\mathbb{F}_q)$ the ring of $\sigma$-dual numbers
over $\mathbb{F}_q$ with elements the pairs in $\mathbb{F}_q\times\mathbb{F}_q$, component-wise addition, and multiplication  defined by
$(x_0,x_1)(y_0,y_1)=(x_0y_0,x_0y_1+x_1\sigma(y_0))$;
\item the Galois ring $\GR(q^2,p^2)=\mathbb{Z}_{p^2}[x]/(f(x))$ where
$f(x)\in\mathbb{Z}_{p^2}[x]$ is a monic polynomial of degree $r$ which is 
irreducible modulo $p$.
\end{enumerate}

In what follows $q, m, N \theta, \eta_i $ will have the
meaning fixed in this section.

\subsection{Modules over Finite Chain Rings}

Let $R$ be a finite chain ring with $|R|=q^m$, where $R/N\cong\mathbb{F}_q$.
The structure of the finitely generated modules over $R$ is well-known
(cf. e.g. \cite{Lang}, Ch. 15, \S2).
It is given by the following theorem.

\begin{theorem}
\label{thm:modules}
Let $R$ be a finite chain ring of length $m$.
For any finite module ${}_RM$ there exists a uniquely 
determined sequence $\lambda=(\lambda_1,\lambda_2,\ldots,\lambda_k)$ with
\[m\ge\lambda_1\ge\lambda_2\ge\ldots\ge\lambda_k\ge1\]
such that ${}_RM$ is a direct sum of cyclic modules:
\[{}_RM\cong R/(\rad{R})^{\lambda_1}\oplus\ldots\oplus R/(\rad{R})^{\lambda_k}.\]
\end{theorem}

We call the sequence $\lambda=(\lambda_1,\ldots,\lambda_k)$ 
the shape  of ${}_RM$. For $\lambda$ we shall also the notation
$m^{a_m}(m-1)^{a_{m-1}}\cdots 1^{a_1}$, where $a_i$ is the number of $\lambda_j$'s 
that are equal to $i$. A module of shape $m^k$ is called a free module. 
The sequence $\lambda'=(\lambda_1',\ldots,\lambda_m')$, where
$\lambda_i'$ is the number of $\lambda_j$'s with
$\lambda_j\ge i$ is called the dual shape of ${}_RM$.
The largest index $k$ with $\lambda_k>0$
is called the rank of ${}_RM$ and is equal to the rank of the 
smallest free module containing ${}_RM$.
The integer $\lambda_m'$ is called the free rank
of ${}_RM$. It is equal to the rank of the largest free submodule contained in ${}_RM$.

There exists a counting formula for the number of submodules of given shape $\mu$ 
contained in a module of fixed shape $\lambda$. It is given by the following theorem.

\begin{theorem}
\label{thm:counting}
Let $R$ be a chain ring of length $m$
with residue field of order $q$.
Let ${}_RM$  be an $R$-module of shape 
$\lambda=(\lambda_1,\ldots,\lambda_n)$. 
For every sequence
$\mu=(\mu_1,\ldots,\mu_n)$, $\mu_1\ge\ldots\ge\mu_n\ge0$, 
satisfying $\mu\le\lambda$ 
(i.e. $\mu_i\le \lambda_i$  for all $i$)
the module ${}_RM$ has exactly
\[\gauss{\lambda}{\mu}{q,m}=\prod_{i=1}^m q^{\mu_{i+1}'(\lambda_i'-\mu_i')}\cdot
\gauss{\lambda_i'-\mu_{i+1}'}{\mu_i'-\mu_{i+1}'}{q}\]
submodules of shape $\mu$. 
Here
\[\gauss{n}{k}{q}=\frac{(q^n-1)\ldots(q^{n-k+1}-1)}{(q^k-1)\ldots(q-1)}.\]
are the Gaussian coefficients.
\end{theorem}

\subsection{Projective Hjelmslev Geometries}

The projective geometries over finite chain rings
(or projective Hjelmslev geometries) are defined in the same way one defines
the classical geometries $\PG(n-1,q)$ over finite fields.
Let $M={}_RR^{n}$, $n\ge3$. Let $\mathcal{P}$ be 
the set of all free rank 1 submodules of $M$, and let
$\mathcal{L}$ be the set of free rank 2 submodules of $M$.
The incidence structure $\Sigma=(\mathcal{P},\mathcal{L})$ with incidence $I$
given by set-theoretical inclusion is called the (left) 
$n$-dimensional projective Hjelmslev geometry.
We denote it by $\PHG(n-1,R)$. Note that two points might be incident with more than one line. Similarly two points can meet in more than one point.

A set of points $\mathcal{S}$ is called a subspace if $P,Q\in\mathcal{S}$ implies  that $\mathcal{S}$ the points of at least one line through $P$ and $Q$.The subspaces of $\PHG(n-1,R)$ are the pointsets associated with the submodules of frank greater than 1.
Subspaces associated with the
free submodules of ${R}R^n$ are called Hjelmslev subspaces. 
Note that the intersection of two Hjelmslev subspaces is a subspace
(or empty), but not necessarily an Hjelmslev subspace.The
projective dimension of a Hjelmslev
subspace $S$ associated with a module of shape $m^s$ is $\dim S=s-1$. 
The shape of a subspace is the shape of the underlying submodule. 

Two points
 $P=Rx$ and $Q=Ry$ are said to be neighbours (notation $P\eye_i Q$)
 if $\eta_i(x)=\eta_i(y)$. More genrally, two subspaces $U$, $V$ of the same shape
 are $i$-neighbors if $\eta_i(U)=\eta_i(V)$. For every subspace $U$, we denote by
 $[U]^{(i)}$ the set of all points that are $i$-neighbors to some point on $U$.
  
The geometries $\PHG(n-1,R)$ have a nice nested structure which is decribed as follows.
Let $\Sigma=(\mathcal{P},\mathcal{L})=\PHG(n-1,R)$. Fix a Hjelmslev subspace $S$ of
(projective) dimension $s-1$. Set
\[\tilde{\mathcal{P}}=\{T\cap[P]^{(m-i)} \mid T\eye_iS,T\cap[P]^{(m-i)}\ne\varnothing\}.\]
Further denote by $\mathcal{L}(S)$ the set of all lines from $\mathcal{L}$
that are contained as a set of points in some Hjelmslev subspace $T$ with $T\eye_i S$.
Further define incidence $\mathfrak{I}\subset\mathfrak{P}\times\mathcal{L}(S)$ by
$(T\cap[P]^{(m-i)},L)\in\mathfrak{I}$ iff $T\cap[P]^{(m-i)}\cap L\ne\varnothing$.
There exist lines in $\mathcal{L}(S)$ that are incident with same sets of points 
from $\mathfrak{P}$. Denote by  
$\mathfrak{L}$ a maximal family of lines from $\mathcal{L}(S)$ that are different
as subsets of $\mathfrak{P}$.

\begin{theorem}
\label{thm:factor_structure}
The incidence structure $(\mathfrak{P},\mathfrak{L},\mathfrak{I})$ can be imbedded isomorphically into
$\PHG(n-1,R/\rad^{m-i}R)$. The missing part is a subspace of shape 
${m-i}^{n-s}(m-i-1)^{s}$ (i.e. a neighbour class $[U]^{(1)}$ where $U$
is a Hjelmslev subspace of $\PHG(n-1,R/\rad^{m-i}R)$ of dimension $n-s-1$).
\end{theorem}
  
\begin{proof}
Let $S$ be given by $x_{s+1}=x_{s+2}=\ldots=x_n=0$. In other words,
$S$ contains all points represented (in homogeneous coordinates)
as vectors orthogonal to the columns of the matrix
$\displaystyle \left(\begin{array}{c}
\vek{0}\\\hline \vek{I}\end{array}\right).$
where $\vek{0}$ is the $k\times(n-k)$ zero matrix, and 
$I$ -- the identity matrix of order $n-k$. Every space $T\in[S]^{(i)}$
can be represented by the space orthogonal to the columns of
\[\left(\begin{array}{c} \vek{A}\theta^i\\\hline \vek{I}
\end{array}\right),\]
where $A=(a_{ij})_{k\times(n-k)}$ and $a_{ij}\in R\setminus\rad^{m-i}R$. If
$\vek{a}_j\theta^i$ is the $j$-th column of $A$ then all the points
in $[S]^{(i)}$ can be reperesente as
\[(b_1,b_2,\ldots,b_k,-\vek{b}\vek{a}_1\theta^i,\ldots,-\vek{b}\vek{a}_{n-k}\theta^i),\]
where $\vek{b}=(b_1,\ldots,b_k)\in ((R\setminus\rad^{m-i}R)^k)^*$ ($R^K)^*$
is the set of all non-torsion vectors,
or all vectors in which at least one component is a unit).
\end{proof}  
  
The incidence structure $(\mathfrak{P},\mathfrak{L},\mathfrak{I})$ 
from Theorem~\ref{thm:factor_structure} will be denoted
by $\Sigma^{(i)}(S)$. The projective Hjelmslev geometry to which 
$\Sigma^{(i)}(S)$ is extended is denoted by
$\overline{\Sigma}^{(i)}(S)\cong\PHG(n-1,R/\rad^{m-i}R)$.
  
For a more in-depth introduction to projective Hjelmslev geometries
we refer to \cite{HL98,HL12}.

\section{Erd\H{o}s-Ko-Rado Type Theorems for systems of Hjelmslev Subspaces}

Let $\Sigma=\PHG(n-1,R)$.
A family $\mathcal{F}$ of subspaces of $\Sigma$ of a fixed shape
$\kappa$ is called $\tau$-intersecting if every two subspaces 
from $\mathcal{F}$ meet in a subspace of shape $\tau$. 

In this  section, we shall be focused on the following two questions:

\begin{enumerate}[(1)]
\item What is the maximal size of a $\tau$-intersecting family of subspaces
of shape $\kappa$ in $\PHG(n-1,R)$?	
\item What is the sttructure of a $\tau$-intersecting familiy in $\Sigma$
that is maximal cardinality?
\end{enumerate}

We are going to tackle the two classical questions for 
$\tau$-intersecting families:

\begin{enumerate}[(1)]
\item  What is the maximal size of a
an $\tau$-intersecting family of subspaces of the same shape 
$\kappa$ in $\PHG(n,R)$?

\item What is the structure of a $\tau$-intersecting family of maximal cardinality?
\end{enumerate}

Throghout this section $R$ will be a chain ring 
with $|R|=q^m$ and $R/\rad R\cong\mathbb{F}_q$. By $\Sigma$ we shall denote
the $(n-1)$-dimensional (left) projetcive Hjelmslev geometry
$\PHG(n-1,R)$, $n\ge3$.

We start with a straightforward lemma which follows from 
the nested structure of the projective Hjelmlsev spaces. It states
that the image under $\eta_i$ of a $\tau$-intersecting
family is a $\tau'$-intersecting family for a  suitable $\tau'$. 

\begin{lemma}
\label{lma:factor}
Let $\mathcal{F}$ be a 
$\tau$-intersecting familiy
of subspaces in $\Sigma$ of shape $\kappa$, where
\[\kappa=m^{k_m}(m-1)^{k_{m-1}}\ldots1^{k_1},\]
\[\tau=m^{t_m}(m-1)^{t_{m-1}}\ldots1^{t_1}.\]
Then for every $i\in\{1,\ldots,m-1\}$
$\eta_i(\mathcal{F})=\{\eta_i(F_1),\cdots,\eta_i(F_M)\}$ is a $\tau'$-intersecting family of
subspaces of shape $\kappa'$ in $\PHG(n-1,R/N^i)$, where
\[\kappa'=i^{k_m}(i-1)^{k_{m-1}}\ldots1^{k_{m-i+1}},\]
\[\tau'=i^{t_m}(i-1)^{t_{m-1}}\ldots1^{t_{m-i+1}}.\]
In particular, $\eta_1(\mathcal{F})$ is a $(t_m-1)$-intersecting family
of $(k_m-1)$-subspaces of $\PG(n-1,q)$.
\end{lemma}

The next result can be viewed as an analogue of Tanaka's Theorem
for intersection families of Hjelmlslev subspaces (shape $\kappa=m^k$)
in the Hjelmslev geometry $\Sigma$.

\begin{theorem}
\label{thm:Tanaka}
Let $\Sigma=\PHG(n-1,R)$, where $|R|=q^m$, $R/N\cong\mathbb{F}_q$.
Let $t,k,n$ be integers with 
$1\le t<k\le n/2$, and let
$\tau=m^t$, $\kappa=m^k$. 
Let further $\mathcal{F}$ be a $\tau$-intersecting family
of subspaces of shape $\kappa$ in $\Sigma$ with the additional property
that the subspaces from $\mathcal{F}$ do have no common points with a neighbor class $[W]$, where $W$ is a Hjelmslev subspace with $\dim W=n-k-1$.
Then 
\begin{equation}
\label{eq:tanaka-phg}
|\mathcal{F}|\le q^{(k-t)(m(n-k-1)+1)}.
\end{equation}
In case of equality, $\mathcal{F}$ is one of the following:

\begin{enumerate}[(a)]
	\item the set of all subspaces through a fixed $(t-1)$-dimensional 
	(Hjlemslev) subspace $(U)$ with 
	$U\cap[W]=\varnothing$;
	
	\item in the case $k=n/2$, 
	$\mathcal{F}$ can also be the set of all
	$(k-1)$-dimensional subspaces on a fixed 
	$(2k-t-1)$-dimensional subspace $U$ with 
	$\dim U\cap W=k-t-1$.
\end{enumerate}
\end{theorem}

\begin{proof}
We shall prove the result by induction on the length $m$ of the chain ring $R$.
the case $m=1$ is just Tanaka's theorem.

Assume the result is proved for intersection sets of Hjelmslev subspaces
in projective Hjelmsev geometries over chain rings of length at most $m-1$.
Let $\mathcal{F}=\{F_i\}_{i=1}^M$ be a  $\tau$-intersection family ($\tau=m^k$)
of $\kappa$-subspaces ($\kappa=m^k$) in $\Sigma$ satisfying the conditions of the theorem.
By Lemma~\ref{lma:factor}, the family
\[\mathcal{F}'=\eta_{m-1}(\mathcal{F})=\{\eta_{m-1}(F_1),\cdots,\eta_{m-1}(F_M)\}\]
is a $\tau'$-intersection family of subspaces 
in $\PHG(n-1,(R/N^m-1))$ of shape $\kappa'$, where
$\tau'=(m-1)^t$, $\kappa'=(m-1)^k$. The images $F_i'=\eta_{m-1}(F_i)$ are $k$-dimensional
subspaces of $\Sigma'$, and, similarly $[F_i']=\eta_{m-1}([F_i])$. 
By the induction hypothesis,
\[|\mathcal{F}'|\le  q^{(n-k)(k-t)+(m-2)(k-t)(n-k-1)}.
\]
The elements of $\mathcal{F}$ the are mapped under
$\eta_{m-1}$ to the same $\kappa'$-subspace in $\Sigma'$ are contained in the same 
neighbor class $[F_u]^{(m-1)}$ of subspaces of shape $\tau'$.

By the structure theorem (Thm~\ref{thm:factor_structure}),
$[F_u]$ has the structure of $\PHG(n-1,R/N)$ minus a subspace of dimension $n-k-1$.
In addition, every two subspaces from $\mathcal{F}$ contained in $[F_u]$ meet
in a subspace that contains  a subspace of shape $\tau$. 
(Actually they meet in a subspace of shape at least $m^t(m-1)^{k-t}$).
The subspaces from $\mathcal{F}$ contained in $[F_u]$ can be viewed as
a set of subspaces in $\PG(n-1,q)$ avoiding a subspace of dimension $n-k-1$, and 
meeting pairwise in a subspace of dimension at least $t$.
Hence their number is at least $q^{(k-t)(n-k-1)}$. Putting this together, we get
\begin{eqnarray*}
|\mathcal{F}| &\le & q^{(k-t)(n-k)+(m-2)(k-t)(n-k-1)}\cdot q^{(k-t)(n-k-1)} \\
              &=& q^{(k-t)((n-k)+(m-1)(n-k-1))}.
\end{eqnarray*}  

Now assume $\mathcal{F}$ is a family of maximal cardinality, i.e. a family for which equality in (\ref{eq:tanaka-phg}) is achieved.

(1) Let $k<\frac{}{}$. We noted that $\eta_{m-1}(\mathcal{F})$ is a $\tau'$-intersection set of
$\kappa'$-spaces in $\PHG(n-1,R/N^{m-1})$. By the induction hypothesis, $\eta_{m-1}(\mathcal{F})$ consists of all $\kappa'$-subspaces through a fixed
$\tau'$-space. The preimage of this common $\tau'$-subspace is a neighbor
class $[T]^{(m-1)}$, i.e. the image of a subspace of shape $m^t(m-1)^{n-t}$.

Let $F_u, F_v\in\mathcal{F}$ with $\dim F_u\cap F_v$and $F_u\nei F_v$, i.e. $\eta_1(F_u)\ne\eta_1(F_v)$.
Clearly, $F_u\cap F_v$ is an Hjelmslev subspace.
The $\kappa$-subspaces of $\mathcal{F}$ contained in
$[F_u]$ are $(k-1)$ subspaces in a geoemtry isomorphic to $\PG(k-1,q)$ avoiding an
$(n-k-1)$-dimensional subspace. Hence they are all subspaces through a fixed
$(t-1)$-dimensional subspace. This $(t-1)$-subspace 
is the image of a space of shape $m^t(m-1)^{k-t}$, $T_u$ say. 
A similar observation can be made for the $\kappa$-subspaces of $\mathcal{F}$ in $[F_v]$
We denote the corresponding space of shape $m^t(m-1)^{k-t}$ by $T_v$.

The intersection $T_u\cap T_v$ should contain an Hjelmslev subspace of dimension $t-1$.
Otherwise it is easy to construct Hjelmslev subspaces of dimension $k$ injelmslev subspace, whence $T_u\cap T_v$  is a $(t-1)$-dimensional Hjelmslev subspace.

No the result is implied by the fact that the graph $G$ with vertices the
$(k-1)$-dimensional subspaces of $\Sigma$ and neighborhood given by $F_uF_v$
is an edge iff $F_u\cap F_v$ is an Hjelmslev subspace of dimension at least $t-1$, is connected.
\smallskip

(2) Now let $k=\frac{n}{2}$. Now $\eta_{m-1}(\mathcal{F})$ is the set of all 
$(k-1)$-dimensional suspaces in a $(2k-t-1)$-dimensional subspace of $\Sigma'$. 
The subspaces from $\mathcal{F}$ in the neighbor class $[F_u]^{(m-1)}$ are either all
canonically intersecting or all of the other type described in Tanaka's theorem.
In the first case, $\mathcal{F}$ is canonically intersecting, while in the second
it consists of all Hjelmslev subspaces of dimension $k-1$ contained in a Hjelmslev
subspace of dimension $2k-t-1$ meeting $W$ in a subspace of dimension $k-t-1$.
\end{proof}

\begin{theorem}
\label{thm:A}
Let $t,k,n$ be integers with $1\le t<k\le n/2$, 
and let $\tau=m^t$, $\kappa=m^k$.
Let $\mathcal{F}$ be a $\tau$-intersecting family of 
$\kappa$-subspaces in $\Sigma=\PHG(n-1,R)$. Then 
	\[|\mathcal{F}|\le\gauss{m^{n-t}}{m^{k-t}}{q^m}=q^{(m-1)(k-t)(n-k)}\gauss{n-t}{k-t}{q}.
	\]
In case of equality, $\mathcal{F}$ is one of the following:

In case of equality $\mathcal{F}$ is one of the following:
\begin{enumerate}[$(a)$]
\item all Hjelmslev subspaces of dimension $k-1$ through a fixed Hjelmslev subspace of dimension $t-1$; 
\item if $2k=n$, all Hjelmslev subspaces of dimension $k-1$ contained in a 
fixed subspace of $\Sigma$ of dimension $2k-t-1=n-t-1$.
\end{enumerate}
\end{theorem}

\begin{proof}
We sketch a proof for rings $R$ with $|R|=q^2$, $R/\rad R\cong\mathbb{F}_q$.
For rings of larger length we use induction on $m$.

Consider an Erd\H{o}s-Ko-Rado set (EKR-set) in $\PHG(n,R)$.
Denote by $\eta$ the canonical homomorphism
\[\eta\colon\left\{\begin{array}{ccc}
R & \to & \mathbb{F}_q \\
a+b\theta & \to & a,
\end{array}\right..\]
where $a,b\in \Gamma$, and $\Gamma$ is a set $q$ elements from $R$ no two of 
which are congruent modulo $\rad R$.
If $U$ and $V$ are Hjlemslev subspaces of dimension $k$ having a common point
then $\eta(U)$ and $\eta(V)$ are subspaces in $\PG(n,q)$ meeting non-trivially.
Hence, $\eta(\mathcal{F})$ is an EKR-set in $\PG(n,q)$ (taking just one copy of
the same subspace).  

Now consider a neighbour class of Hjelmslev $k$-subspaces $[T]$ containing at least one subsace from 
$\mathcal{F}$. It is known that the geometry having as points the subspaces
$U\cap[x]$, where $U\in[T]$ and $x$ is a point that has a neighbor on $T$ is isomorphic
to $\PG(n,q)$ from which a $(n-1-k)$-dimensional subspace is deleted.
We denote this structure by $\Delta$. Now the largest set of $k$-dimensional subspaces in 
$\Delta$ every two of which meet in a point  consists of all $k$-subspaces
in $\Delta$ through a fixed point,or, in other words the maximal EKR-system
in $\PG(n,q)$ consisting of $k$-subspaces not meeting a fixed $(n-k-1)$-space $V$
consists of all subspaces not meeting $V$ containig a fixed point (Theorem~\ref{thm:tanaka}).

Now consider two neighbor classes $[S]$ and $[T]$ where $S$ and $T$ are Hjelmslev $k$-subspaces.
If the subspaces from $\mathcal{F}\cap[S]$ have the common point $x$ and all subspaces
from $\mathcal{F}\cap[T]$ have the common point $y$ then one can easily
construct two subspaces, from $\mathcal{F}\cap[S]$ and $\mathcal{F}\cap[T]$ respectively,
that do not have a common point. Hence $x=y$ which proves the result.
\end{proof}


\section{Erd\H{o}s-Ko-Rado Families of Non-free Subspaces}

The following example gives an intersecting family of non-free
subspaces which is not canonically intersecting.
Consider a chain ring $R$ with $|R|=q^2$, and with $R/\rad R\cong\mathbb{F}_q$. Let
$\Sigma=\PHG(3,R)$ be the threedimensional left projective Hjelmslev geometry
over $R$.
Set $\lambda=(2,2,1,0)$, i.e. the subspaces of shape $\lambda$ are the line 
stripes consisting of $q^2(q+1)$ points each. 
We are looking for an EKR-family of 
subspaces of shape $\lambda$ of the largest possibl size.
 
First consider the family $\mathcal{F}$ of
all $\lambda$-subspaces through a fixed point in $\Sigma$.
It is easily checked that
\[|\mathcal{F}|=q(q+1)(q^2+q+1).\]
Interestingly, there exist larger EKR-families.

First we are going to establish an upper bound on the size of an 
intersecting family of $\lambda$-subspaces. Assume $\mathcal{F}$
is such a family. Obviously, the set of all neighbour classes of lines containing
subspaces from $\mathcal{F}$ is an intersection family of lines in the
factor geometry (isomorphic to $\PG(3,q)$) and hence has size
$q^2+q+1$. In each neighbour class of lines
there exist $q^2(q+1)$ $\lambda$-subspaces, 
and every two such subspaces meet. Hence we get
\[|\mathcal{F}|\le q^2(q+1)(q^2+q+1).\]
However this bound can be improved.

Consider two $\lambda$-subspaces from $\mathcal{F}$, $F_1$ and $F_2$ say, that belong
to different neighbour classes of lines, but lie in planes that are neighbours.
Denote the latter by $H_1$ and $H_2$; obviously $[H_1]=[H_2]=:[H]$.
Then $[H]$ has the structure of $\PG(3,q)$ minus a point 
(dual affine space) and $F_1$, $F_2$ are lines in it. This implies that $[H]$ contains at most
$q(q+1)$ stripes from $\mathcal{F}$ and each of $[F_1]$ and $[F_2]$ contains at most 
$q$ stripes from $\mathcal{F}$. 

Let $F\in\mathcal{F}$ and denote by $\nu_{[F]}$ 
the number of planes in the factor geometry determined by $[F]$ and another
class of lines containig subspaces from $\mathcal{F}$. Now the number of $\lambda$-subspaces 
contained in $[F]$  is at most
\begin{equation}
\label{eq:stripes}
q^2(q+1-\nu_{[F]})+q\nu_{[F]}\le q(q^2+1).
\end{equation}
This gives the slightly better estimate on the number of $\lambda$-subspaces
in an intersection family $\mathcal{F}$:
\begin{equation}
\label{eq:Fbound}
|\mathcal{F}|\le q(q^2+1)(q^2+q+1).
\end{equation}  
 
Now we shall demonstrate that the bound (\ref{eq:Fbound}) is sharp.
In order to construct an intersection family of this size
we should start with a maximal intersection family in the factor geometry $\PG(3,q)$.
It is one of the following

- a pencil of lines,

- all lines in a plane.

In the first case we have $\nu_{[F]}=q+1$ for every $F\in\mathcal{F}$
and the bound cannot be reached since the inequality in (\ref{eq:stripes}) is sharp.
In second case, we have $\nu_{[F]}=1$ for all $F\in\mathcal{F}$ and an intersection family of the size
given by (\ref{eq:Fbound}) can be constructed.
Such a family consists of all $\lambda$ subspaces
contained ina fixed neighbour class of planes $[H]$ except for those that have the direction 
of $[H]$. Among the latter we select only $q(q+1)$  that form an
intersection family in the factor geometry defined on $[H]$, isomorphic to
$\PG(3,q)$ minus a point.  

\begin{theorem}
	Let $R$ be finite chain ring with $|R|=q^2$, $R/N\cong\mathbb{F}_q$.
Let $k\ge1$ be  an integer and let
$\tau=2^1$, $\kappa=2^k1^{k-1}$, and
$n=2k$. Let $\mathcal{F}$ be a $\tau$-intersecting family of 
$\kappa$-subspaces in $\Sigma=\PHG(2k-1,R)$. Then
	\[|\mathcal{F}|\le \left(q^{k+1}\gauss{k-1}{1}{q}+1\right)\gauss{2k-1}{k-1}{q}.
	\]
In case of equality, $\mathcal{F}$ is the following:

\begin{enumerate}[]
	\item all subspaces of shape $\kappa$ contained in $[F]$,
	where $F$ is a hyperplane in $\Sigma$, 
	apart from those that have the ``direction'' 
	of $F$, plus all $\kappa$-subspaces contained in 
	$F$.	
\end{enumerate}

\end{theorem}

\section*{Acknowledgements}
This research of the first author was supported 
by the Bulgarian  National  Science  Research  Fund
under  Grant  KP-06-N72/6-2023. The research of the third author was supported by 
the
Research Fund of Sofia University under Contract 80-10-14/21.05.2025.

\end{document}